\documentclass[12pt,oneside]{amsart}

\usepackage{amsthm,amsmath,amssymb}
\usepackage{mathrsfs}
\usepackage{enumerate}
\usepackage{amsfonts}
\usepackage{verbatim}
\usepackage{amsbsy}
\usepackage{amsmath}
\usepackage{amssymb}
\usepackage{MnSymbol}
\usepackage{mathrsfs}

\newtheorem{theorem}{Theorem}[section]
\newtheorem{lemma}[theorem]{Lemma}

\newtheorem{conjecture}[theorem]{Conjecture}
\newtheorem*{burning number conjecture}{Burning number conjecture}
\newtheorem*{first main theorem}{First Main Theorem}
\newtheorem*{second main theorem}{Second Main Theorem}

\newtheorem*{claim}{Claim}

\theoremstyle{definition}
\newtheorem{definition}[theorem]{Definition}

\theoremstyle{remark}
\newtheorem{remark}[theorem]{Remark}

\newtheorem{case}{Case}
\newtheorem{subcase}{Case}
\numberwithin{subcase}{case}
\newtheorem{special case}{Special Case}

\subjclass[2000]{05C85, 05C82}

\begin{document}

\title[Burning Numbers of Path Forests and Spiders]{Graph Burning: Tight Bounds on the Burning Numbers of Path Forests and Spiders}
\author{Ta Sheng Tan}
\address{Institute of Mathematical Sciences\\
University of Malaya\\  
50603 \linebreak
Kuala Lumpur, Malaysia}
\email{tstan@um.edu.my}
\author{Wen Chean Teh\! $^*$}
\address{School of Mathematical Sciences\\
	Universiti Sains Malaysia\\
	11800 USM,\linebreak
	Malaysia}
\email{dasmenteh@usm.my}
\keywords{Burning number conjecture; Graph algorithm; Spider; Path forest}

\begin{abstract}
In 2016, Bonato, Janssen, and Roshanbin introduced graph burning as a  discrete process that models the spread of social contagion. Although the burning process is a simple algorithm, the problem of determining the least number of rounds needed to completely burn a graph, called the burning number of the graph, is \textbf{NP}-complete even for elementary graph structures like spiders. An early conjecture that every connected graph of order $m^2$ can be burned in at most $m$ rounds is the main motivator of this study. Attempts to prove the conjecture have resulted in various  upper bounds for the burning number and validation of the conjecture for certain elementary classes of graphs. In this work, we find a tight upper bound for the order of a spider for it to be burned within a given number of rounds. Our result shows that the tight bound depends on the structure of the spider under consideration, namely the number of arms. This strengthens the previously known results on spiders in relation to the conjecture.
More importantly, this opens up potential enquiry into the connection between burning numbers and certain characteristics of graphs. Finally, a tight upper bound for the order of a path forest for it to be burned within a given number of rounds is obtained, thus completing previously known partial corresponding results.
\end{abstract}

\maketitle

\let\thefootnote\relax\footnotetext{$^*$ Corresponding author}

\section{Introduction}
Graph burning is a process that models the spread of social contagion \cite{bonato2016how, roshanbin2016burning}.
It is  a discrete-time deterministic process. Suppose $G$ is a simple finite graph. Initially, every vertex of $G$ is \emph{unburned}. At the beginning of every \emph{round} $t\geq 1$,
a \emph{burning source} is place at an unburned vertex, turning its status to \emph{burned}. A burned vertex remains burned  throughout the process. If a vertex is already burned at round $t-1$, then its unburned neighbors (if any) become automatically \emph{burned} at (the end of) round $t$.  The burning process is completed when all vertices are burned. The \emph{burning number of $G$}, denoted by $b(G)$, is the least number of rounds needed  for the burning process to be completed.
Equivalently, $b(G)$ is the least number $m$ such that the set of vertices of $G$ can be covered by $m$ balls of radii $0,1,2, \dotsc, m-1$ respectively, where any vertex of $G$ at graph distance at most $r$ from $v$ can be covered by the
ball of radius $r$ placed at $v$. 
We say that a graph is \emph{$m$-burnable} if its burning number is at most $m$.


The study on graph burning can be focused on trees because for any connected graph $G$, we see from \cite{bonato2016how} that
$$b(G)=\min \{ \, b(T)\mid T \text{ is a spanning tree of } G \,\}.$$
Furthermore, if $T'$ is a subgraph  of a tree $T$, then $b(T')\leq b(T)$, which is generally not true if $T$ were not a tree \cite{bonato2016how}. 
In fact, even elementary tree structures like spiders can be deemed complicated  because it is \mbox{\textbf{N\!P}-complete} \cite{bessy2017burning}
to compute their burning numbers.

In this exposition, unless stated otherwise, lower case variables are assumed to be integer-valued and positive.

\begin{theorem}\cite{bonato2016how}\label{2803a}
	The burning number of every path of order $m$ is $\lceil \sqrt{m}\rceil$.	
\end{theorem} 

Bonato, Janssen, and Roshanbin \cite{bonato2016how} showed that the burning number of any connected graph of order $m$ is bounded above by $2\sqrt{m}-1$.
In the same exposition, they made the following conjecture, which remains open.

\begin{burning number conjecture}
The burning number of every connected graph of order $m$ is at most $\lceil \sqrt{m}\rceil$.
\end{burning number conjecture}

Attempts to prove the burning number conjecture result in improved upper bounds.
Bessy, Bonato, Janssen, Rautenbach, and Roshanbin \cite{bessy2018bounds} proved that $$b(G)\leq \left(\sqrt{\frac{32}{19}}+o(1)\right)\sqrt{m}$$
for every connected graph $G$ of order $m$.
Later, Land and Lu \cite{land2016upper} improved this bound to $\left\lceil \frac{-3+\sqrt{24m+33}}{4}\right\rceil$.
However, the burning number conjecture remains open. Meanwhile, 
other authors have taken a different approach and verified the burning number conjecture for some classes of graphs, including, the generalized Petersen Graphs \cite{sim2018burning}, the hypercube graphs \cite{mitsche2018burning}, and the circulant graphs \cite{fitzpatrick2017burning}. 
Unsurprisingly, the burning number conjecture has also elevated the interests in the class of spiders, which is the main focus of this study.

Suppose $n\geq 3$. 
An \emph{$n$-spider} is a tree with exactly one vertex of degree $n$, called the \emph{head} of the spider, and every other vertex has degree at most two.
The paths from the head to the leaves of the spider are called \emph{arms}. The distance along an arm from the head to its leaf is its \emph{arm length}, which equals the number of vertices on the arm excluding the head.
The vertex next to the head of an arm is considered the first vertex of the arm and so forth.
We say that an $n$-spider is \emph{balanced} if all its arms have the same length.
For our induction purposes, we also regard a path of length $l$ as a $2$-spider with arm lengths $l_1$ and $l_2$ (not unique) such that $l_1+l_2=l$.

\begin{theorem}\cite{bonato2018bounds, das2018burning}\label{2703a}
Every spider of order at most $m^2$ is  $m$-burnable.	
\end{theorem}	

In this study, we strengthen Theorem~\ref{2703a} by finding an exact tight bound $N(n,m)$ such that every $n$-spider of order at most $N(n,m)$ is $m$-burnable. Our work is motivated by the instinct that the more deviation a tree is from a path, the less likely for it to violate the burning number conjecture. More precisely, the more arms a spider has, the larger its order has to be before 
it becomes unburnable within a given number of rounds. Before we spell out our main result, we give the following definition.

\begin{definition}
For every $m\geq 2$ and $n\geq 2$, let $I^{sp}_{n,m}$ denote the largest integer $N$ such that every $n$-spider of order at most $N$ is $m$-burnable.
\end{definition}

 For example, $I^{sp}_{3,2}=5$ and $I^{sp}_{3,3}=9$. 
 By Theorem~\ref{2803a}, $I_{2,m}=m^2$ for all $m\geq 2$.
 Since the $n$-spider of order $m^2+n-1$
 with arm lengths $m^2-1, \underbrace{1,1, \dotsc, 1}_{ n-1 \text{ times}}$
 contains the path of order $m^2+1$ as a subtree, it is not $m$-burnable. Hence, by Theorem~\ref{2703a}, 
 $m^2\leq I^{sp}_{n,m}\leq m^2+n-2$ for every integer $m,n\geq 2$. Our first main result completely ascertains the values $I^{sp}_{n,m}$.

\begin{first main theorem}
Let $n\geq 2$. Then
\begin{enumerate}
	\item $I^{sp}_{n,m}= m^2+n-2$ for all $m>n$;
	\item  $I^{sp}_{n,m}= m^2+n-m$ for all $2\leq m\leq n$.	
\end{enumerate}	
\end{first main theorem}


\begin{remark}
	If one of the arm of an $n$-spider (for $n>2$) has length one, then deleting that arm results in a subspider with the same burning number. 
	\end{remark}

It is conceivable that, for most spiders, 
the most efficient burning strategy  would involve placing  the first burning source at the head. When this is the case, the remaining vertices unburned by the first source in the whole burning process form 
a disjoint union of independent paths, which is called \emph{path forest}. 
Hence, the study of the burning of spiders is closely connected to that of path forests.
 By \emph{path orders} of a path forest, we mean  the respective numbers of vertices in each of its paths.


The earliest result regarding path forests says that $b(T) \leq \sqrt{\vert T\vert}+n-1$ for any path forest $T$ with $n$ paths \cite{bonato2016how}.
Bonato and Lidbetter \cite{bonato2018bounds} obtained two new upper bounds, which together 
provide improvement over the previous bound. Independently, Das, Dev, Sadhukhan, Sahoo, and Sen \cite{das2018burning}  obtained an exact tight bound on the order of certain path forest for it to be burned within a given number of rounds.

\begin{theorem}\cite{das2018burning}\label{2703c}
Let $m\geq n\geq 2$ and suppose $T$ is a path forest with $n$ paths such that all but possibly one of the paths have order at most $m$. If 
$$\vert T\vert \leq m^2-(n-1)^2$$	
	then $T$ is $m$-burnable.
	\end{theorem}

The upper bound is tight because the path forest 
of order $m^2-(n-1)^2+1$
and path orders  $\underbrace{2,2,\dotsc, 2}_{n-1 \text{ times}}, m^2-n^2+2$ is not $m$-burnable.	
We suspected that apart from this unique path forest, every other path forest of order $m^2-(n-1)^2+1$
is $m$-burnable. This leads to our second main result, which says that the condition in Theorem~\ref{2703c} placed on the path forests can be lifted completely. 

\begin{definition}
	For every $m\geq n\geq 2$, let $I^{p\!f}_{n,m}$ denote the largest integer $N$ such that every path forest with $n$ paths of order at most $N$ is $m$-burnable.
\end{definition}

\begin{second main theorem}\label{2803c}
$I^{p\!f}_{n,m}= m^2-(n-1)^2$ for all $m\geq n\geq 2$.
\end{second main theorem}

We end the introduction by highlighting some of the works on graph burning.
Bounds on the burning numbers of Cartesian products, strong products, and lexicographic products of graphs have been studied in \cite{mitsche2018burning}.
 Randomness is injected into the study of graph burning in \cite{mitsche2017burning}. 
 A variance of the burning process, where the underlying graph may evolve over time, were recently introduced by Bonato, Gunderson, and Shaw \cite{bonato2018burning2}. Finally, since the graph burning problem is \textbf{N\!P}-complete,
 it has triggered the pursuit of polynomial time approximation algorithms for the burning numbers of classes of graphs \cite{bessy2017burning, bonato2018approximation, bonato2018bounds}.

\section{Main Result on Spiders} \label{0701a}  

In this section, we prove a series of lemmas and theorems leading to our First Main Theorem. The technical lemmas deal with path forests (Lemma~\ref{0405a} and Lemma~\ref{0603c}) and spiders (Lemma~\ref{0403a} and Lemma~\ref{1804a})   with certain specified structures. First, we make a simple observation that will be repeatedly used in the proofs throughout the paper.

\begin{remark}\label{1604a}
During a burning process, a single burning source can be used to burn a path of order (at most) $2m-1$ in $m$ rounds, by placing	this burning source at the center of the path.	
\end{remark}

\begin{lemma}\label{0405a}
Let $n\geq 2$ and suppose $T$ is a path forest with path orders $l_1\geq l_2\geq \dotsb \geq l_n$ such that $l_{n}=1$ and $l_{n-1}\geq 3$. If the order of $T$ is at most $4n-4$, then $T$ is $n$-burnable.
\end{lemma}

\begin{proof}
We argue by induction. The base step $n=2$ is straightforward. For the induction step, suppose $n>2$ and the lemma is true for $n-1$. We may assume that $T$ is a path forest of order $4n-4$ with path orders $l_1\geq l_2\geq \dotsb \geq l_n$ such that $l_{n}=1$ and $l_{n-1}\geq 3$. Then it can be deduced that $4\leq l_1\leq n+1$. Let $T'$ be the path forest obtained by deleting the first path of $T$. Since $\vert T'\vert\leq 4(n-1)-4$, by the induction hypothesis, $T'$ is $(n-1)$-burnable. Therefore, $T$ is $n$-burnable by placing the first burning source at the/a center of the first path of $T$ and then follow the burning sequence of $T'$.
\end{proof}

\begin{lemma}\label{0403a}
Let $m>n\geq 2$ and suppose $T$ is an $n$-spider of order  $m^2+n-2$ such that every arm length is at most $2m-1$. Then $T$ is \mbox{$m$-burnable}.
\end{lemma}

\begin{proof}
Let $T$ be an $n$-spider as in the lemma and let $2m-1\geq l_1\geq l_2\geq \dotsb \geq l_n$ denote the arm lengths of  $T$. So we have $l_1+l_2+\dotsb + l_n = m^2+n-3$. 
If $l_i \leq (m-1)+(2m-1-2i)$ for all $1\leq i\leq n$, then it is straightforward that $T$ is $m$-burnable. Indeed, we can place the first burning source at the head of $T$, which will burn the first $m-1$ vertices of each arm in $m$ rounds.
For each $1\leq i\leq n$, to burn the remaining path of order at most $2m-1-2i$ on the $i$-th arm, we place the $(i+1)$-th burning source at the center of this path (if this path is nonempty). From Remark~\ref{1604a}, we deduce that these $n+1$ burning sources are sufficient to burn $T$ in $m$ rounds.

Hence, we may assume there is a least  $k$ such that 
		$$l_{k} > (m-1)+ (2m-1-2k)= 3m-2-2k.$$
	Since $l_{k}\leq 2m-1$, it follows that $2k\geq m$. 
	
	Let $l'_i= l_i- (3m-2-2k)$ for all $1\leq i\leq k$. Note that
	$$l'_1+l'_2+\dotsb+l'_{k}+ l_{k+1}+l_{k+2}+\dotsb +l_n = m^2+n-3-k(3m-2-2k),$$
$l'_1\geq l'_2\geq \dotsb\geq l'_{k}\geq 1$, and
$l_{k+1}\geq l_{k+2}\geq \dotsb \geq l_n\geq 1$.
		The following is our key technical observation.
	$$m^2+n-3-k(3m-2-2k) \leq n+m-3 \qquad (\ast)$$

	Assume $(\ast)$ holds and we consider a burning process using $m$ rounds where the first burning source is placed at the head of $T$.
	Then $l_{k+1}+l_{k+2}+\dotsb +l_n \leq n+m-3-k$. It follows that $l_{k+1}\leq m-2$ and thus the last $n-k$ arms (provided, $k<n$) would be completely burned by the first source at the head in $m$ rounds.
		
	On the other hand,
$l'_1+l'_2+\dotsb+l'_{k}\leq n+m-3-(n-k)= m-3+k$.
	If $l'_{k}\geq 3$, then as $2k\geq m$, we have
$l'_1+l'_2+\dotsb+l'_{k}\geq 3k >m-3+k$, which is a contradiction. Hence, $l'_{k}\leq 2$.

By our choice of $k$ and a burning process similar to the one before,  it follows that the first $k+1$ burning sources would burn the entire $T$ in $m$ rounds, except possibly the last $l'_{k}$ vertices of the $k$-th arm. In the scenario when $m\geq n+3$, or when $m=n+2$ with $k\leq  n-1$, or when $m=n+1$ with $k\leq n-2$, there are at least two rounds left after the first $k+1$ rounds. Hence, placing the $(k+2)$-th burning source at the last vertex of the $k$-th arm suffices to complete the burning process since $l'_{k}\leq 2$.
	
	Therefore, there are three remaining special cases.

\setcounter{special case}{0}

\begin{special case}
	$m=n+2$ and $k=n$.
	
	 In this case,  $l'_1+l'_2+\dotsb+l'_n =n+1$ and thus $l'_n=1$. Hence, this sole vertex unburned by the first $n+1$ sources can be burned by the last remaining source.
\end{special case}

\begin{special case}
	$m=n+1$ and $k=n-1$.
	
In this case,   $l'_1+l'_2+\dotsb+l'_{n-1}+l_n =n+1$  and $n\geq 3$ (since $n-1=k\geq \frac{m}{2}\geq \frac{3}{2}$), implying that $l'_{n-1}=1$. Similarly, this sole vertex unburned by the first $n$ sources can be burned by the last remaining source.
\end{special case}

\begin{special case}
$m=n+1$ and $k=n$.

In this case,  $l'_1+l'_2+\dotsb+l'_n =2n-2$. Although $l'_n=1$, there is no more remaining burning source and thus a different burning process is necessary. Note that $l_n= l_n'+(3m-2-2k)=n+2$.
We consider a burning process of $T$ where the first burning source is placed at the first vertex on the last arm. This first source will burn the first $n+1$ vertices of the last arm and the first $n-1$ vertices of each of the first $n-1$ arms in $n+1$ rounds. Let $l''_i=l_i-(n-1)$ for each $1\leq i\leq n-1$ and $l''_n=1$. Then the vertices remain unburned by the first souce in $n+1$ rounds form a path forest $T''$ with path orders
$l''_1\geq l''_2\geq \dotsb \geq l''_n$ such that $l''_n=1$ and $l''_{n-1}\geq 3$. Since
$$\vert T''\vert = (n+1)^2+n-3-(n-1)(n-1)-(n+1)= 4n-4,$$
by Lemma~\ref{0405a}, $T''$ is $n$-burnable. Hence, $T$ is $(n+1)$-burnable by following the burning sequence of $T''$ from the second burning source onwards.
\end{special case}
	
	Therefore, it remains to prove $(\ast)$, which is equivalent to the following claim.
	
	\begin{claim}
		$k(3m-2-2k) \geq m^2-m$.
	\end{claim}
	
		\begin{proof}[Proof of the claim]
	Note that $k\leq n<m$ and recall that $2k\geq m$. Hence, we have  
	$\left\lfloor \frac{m+1}{2}\right\rfloor \leq    k\leq m-1$. By observing that $k(3m-2-2k)$ is a quadratic function in $k$, it is straightforward to check that the function is minimized at $k=m-1$ within the interval $\left[ \left\lfloor \frac{m+1}{2}\right\rfloor, m-1    \right]$.
	Hence, $k(3m-2-2k)\geq (m-1)[ 3m-2-2(m-1)]= m^2-m$.
		Therefore, the claim follows.\renewcommand{\qedsymbol}{}
\end{proof}
The proof is now complete.
\end{proof}

\begin{remark}\label{1604c}
From the proof of Lemma~\ref{0403a}, it can be seen that in the burning process of $T$ in $m$ rounds, the head of $T$ is burned in either the first round or the second round.	
\end{remark}	


\begin{lemma}\label{0603c}
Let $n\geq 1$ and suppose $T$ is a path forest with $n$ paths. If $m\geq n$ and $T$ has order at most $3m-1-n$, then $T$ is $m$-burnable.
\end{lemma}

\begin{proof}
We argue by induction on $n$. The base step $n=1$ is straightforward. 
For the induction step, suppose $n>1$ and the lemma is true for $n-1$.
If the longest path of $T$ has order one, then since $n\leq m$, it is trivially $m$-burnable.
Else if the longest path of $T$ has order between $2$ and $2m-1$, then by deleting a/this longest path, we get 
a path forest $T'$ with $n-1$ paths and order at most $3m-1-n-2=3(m-1)-1 -(n-1)$; hence, by the induction hypothesis, $T'$ is $(m-1)$-burnable and thus $T$ is $m$-burnable. Otherwise, the longest path of $T$ has order more than $2m-1$. Again, by deleting this longest path, we get a path forest $T''$ with $n-1$ paths and order at most $3m-1-n-2m=m-1-n$. 
Note that $m-1-n\geq n-1$, implying that $m\geq 2n \geq n+2$. Also, note that
$(2m-3)+(2m-1)=4m-4$ is larger than the order of the longest path. Since $m-2 \geq n-1$ and (because $m\geq 4$)
$$\vert T''\vert \leq m-1-n\leq 3(m-2)-1-(n-1),$$  by the induction hypothesis,
$T''$ is $(m-2)$-burnable and thus $T$ is $m$-burnable.
\end{proof}

\begin{theorem}\label{0403b}
	Let $n\geq 2$. Every $n$-spider of order at most $n^2+3n-1$ is \linebreak $(n+1)$-burnable. Furthermore, if $l$ is the length of the shortest arm, then the spider can be burned in $n+1$ rounds in such a way that after the head of the spider is burned, there are still at least $\min\{l,n-1\}$ rounds.	
\end{theorem}

\begin{proof}
We argue by induction $n$. For the base step $n=2$, it  is straightforward to check that every $2$-spider  of order at most nine can be burned in three rounds in such a way that its head is not burned last.
For the induction step, 
	suppose $n>2$ and  the theorem is true for $n-1$.

 Suppose $T$ is an $n$-spider with arm lengths $l_1\geq l_2\geq \dotsb \geq l_n$ such that
$\vert T\vert \leq n^2+3n-1$.
	If $l_1\leq 2n+1$, then we are done by Lemma~\ref{0403a} and Remark~\ref{1604c}. Hence, assume $l_1>2n+1$. Let $l_1'=l_1-(2n+1)$.

\setcounter{case}{0}

\begin{case}
$l_n \leq n-2$.

\begin{subcase}
$l_1'\geq l_n$.

Let $T'$ be the $(n-1)$-subspider of $T$ with arm lengths
$l'_1, l_2, \dotsc, l_{n-1}$.
Since 
$$\vert T'\vert \leq  n^2+3n-1 -(2n+1)-1= (n-1)^2+3(n-1)-1,$$
by the induction hypothesis, $T'$ can be burned in $n$ rounds in such a way that there are still at least 
$\min\{ l_1', l_{n-1}, n-2  \} $ rounds left after the head of the spider is burned. To see that $T$ is $(n+1)$-burnable, we burn the last $2n+1$ vertices on the first arm of $T$ by the first burning source. Since $\min\{ l_1', l_{n-1}, n-2  \}\geq l_n $, the remaining vertices of $T$ can be burned according to the burning sequence of $T'$. 
\end{subcase}

\begin{subcase}
$l_1'< l_n$.

Let $T''$ be obtained by deleting the longest arm from $T$. Since
$$\vert T''\vert \leq n^2+3n-1 -(2n+2)= (n-1)^2+3(n-1)-1,$$ by the induction hypothesis, $T''$ 
can be burned in $n$ rounds in such a way that there are still at least 
$\min\{ l_n, n-2  \} $ rounds left after the head of the spider is burned.
To see that $T$ is $(n+1)$-burnable, again we burn the last $2n+1$ vertices on the first arm of $T$ by the first burning source. Since $\min\{ l_n,  n-2  \}=l_n> l_1' $, the remaining vertices of $T$ can be burned according to the burning sequence of $T''$.
\end{subcase}
\end{case}

\begin{case}
	$l_n\geq n-1$.
	
We consider a burning process with $n+1$ rounds where the first burning source is placed at the head of $T$. 

\begin{subcase}
$l_1'\leq n-2$.

This implies that $l_1\leq (n-2)+(2n+1)= n+(2n-1)$.
Hence, using the second source together with the first source placed at the head of $T$, the first arm can be completely burned. Let $k$ be the number of the \emph{other} arms with length more than $n$. (If $k=0$, then $T$ would be fully burned at the end of the burning process by the first two burning sources.) Assume $k$ is positive. Then  the vertices remain unburned by the first two sources form a path forest with $k$ paths and order at most
$$n^2+3n-2-kn -(n-k-1)(n-1)-(2n+2)=3n-5-k< 3(n-1)-1-k.$$
 Hence, by Lemma~\ref{0603c}, this path forest is $(n-1)$-burnable and therefore $T$ is  $(n+1)$-burnable.
\end{subcase}

\begin{subcase}
$l_1'\geq n-1$.

This implies that $l_1\geq 3n$.
Let $k$ be the number of the other arms with length more than $n$. Note that $k\neq n-1$ or else $l_1\leq n^2+3n-2-(n-1)(n+1)=3n-1$, a contradiction. 
Using the second source together with the first source, the first $3n-1$ vertices of the first arm would be burned at the end of the burning process. Then the vertices remain unburned by the first two sources form a path forest with $k+1$ paths (possibly $k=0$) and order at most
$$n^2+3n-2-kn -(n-k-1)(n-1)-(3n-1)=2n-2-k.$$
Since $n\geq 3$, it follows that
$$2n-2-k\leq 3n-5-k=3(n-1)-1-(k+1).$$
Therefore, noting that $k+1\leq n-1$,  this path forest is $(n-1)$-burnable 
by Lemma~\ref{0603c} and thus $T$ is $(n+1)$-burnable.
\end{subcase}
Therefore, the above cases complete the proof of the lemma.\qedhere
\end{case}
\end{proof}

\begin{theorem}\label{1103c}
	Let $m>n\geq 2$. Every $n$-spider of order at most $ m^2+n-2$ is $m$-burnable. 
\end{theorem}

\begin{proof}
	Fix an arbitrary $n\geq 2$.	We argue by induction on $m$. The base step $m=n+1$ holds by Theorem~\ref{0403b}. Now, suppose $m>n+1$ and the theorem holds for $m-1$. Suppose $T$ is an $n$-spider of order at most $m^2+n-2$ with arm lengths $l_1\geq l_2\geq \dotsb \geq l_n$.
	If $l_1\leq 2m-1$, we are done by Lemma~\ref{0403a}. Hence, assume $l_1>2m-1$. Let $l_1'=l_1-(2m-1)$ and let $T'$ be the $n$-spider with arm lengths $l'_1, l_2, \dotsc, l_n$. Since $\vert T'\vert \leq m^2+n-2-(2m-1)= (m-1)^2+n-2$,
	by the induction hypothesis, $T'$ is $(m-1)$-burnable.
	Now, by placing the first burning source at the center of the last $2m-1$ vertices of the first arm of $T$, and then follow the burning sequence of $T'$, we deduce that $T$ is $m$-burnable.	
\end{proof}

\begin{lemma}\label{1804a}
Let $m\geq 2$. The balanced $m$-spider of order $m^2+1$ is not $m$-burnable.
\end{lemma}

\begin{proof}
The case $m=2$ is trivial.
Suppose $m\geq 3$. We argue by contradiction. Assume the balanced $m$-spider $T$ with each arm length being $m$ is $m$-burnable. 
Note that any two arms together with the head of the spider forms a path of order $2m+1$. Hence, in a burning process of $m$ rounds, if any burning source placed at this path burns one of the end vertices, then at least two vertices at the other end of the path would not be burned by this source. It follows that in $m$ rounds, assuming $m-1$ leaves are burned by the first $m-1$ burning sources, the last two vertices of the arm with the remaining leaf would not be burned. However, it would be impossible to complete the burning process using the last remaining burning source.
\end{proof}

\begin{theorem}\label{2702b}
	Let $n\geq m \geq  2$. If $T$ is an $n$-spider of order at most $m^2+n-2$, then $T$ is $m$-burnable unless 
	$T$ contains the balanced $m$-spider of order $m^2+1$ as a subtree (when $n \geq 3$).	
\end{theorem}

\begin{proof}
The case $n=2$ is straightforward. Let $n\geq 3$ and $2\leq m\leq n$. Suppose $T$ is an $n$-spider of order at most $m^2+n-2$ with arm lengths $l_1\geq l_2\geq \dotsb \geq l_n$. Assume $T$ does not contain the balanced $m$-spider of order $m^2+1$ as a subtree. This implies that $l_m\leq m-1$. 	

\setcounter{case}{0}

\begin{case}
	$l_m\leq m-2$.
	
	Let $T'$ be the $(m-1)$-subspider of $T$ consisting of the first $m-1$ arms of $T$ together with the head. Note that $m-1\geq 2$ in this case and
	$$\vert T'\vert \leq m^2+n-2 -(n-m+1)= m^2+m-3= (m-1)^2+3(m-1)-1.$$
Hence, by Theorem~\ref{0403b}, $T'$ can be burned in $m$ rounds in such a way that after the head of $T'$ is burned, there are still at least $\min\{l_{m-1},m-2\}$ rounds left. Since $\min\{l_{m-1},m-2\}\geq l_m \geq \dotsb \geq l_n$, it follows that $T$ is $m$-burnable. 
	\end{case}

\begin{case}
$l_m= m-1$.

We consider a burning process with $m$ rounds where the first burning source is placed at the head of $T$. Then the first $m-1$ vertices of each of the first $m$ arms and the entire last $n-m$ arms would be burned at the end of the process by the first source. Hence, the number of  vertices unburned by the first source is at most
$$m^2+n-3- m(m-1)- (n-m) = 2m-3$$
and they form a path forest $T''$ with $k$ paths, where $1\leq k\leq m-1$. Since
$$2m-3 \leq 2m-3 + (m-1-k)=3(m-1)-1-k,$$  
it follows by Lemma~\ref{0603c} that $T''$ is $(m-1)$-burnable and thus $T$ is $m$-burnable.
\end{case}
Therefore, the proof is complete.
\end{proof}

\begin{proof}[Proof of First Main Theorem]
Fix an arbitrary integer $n\geq 2$.
Suppose $m>n$. The \mbox{$n$-spider} with arm lengths $m^2-1, \underbrace{1,1, \dots, 1}_{n-1 \text{ times}}$ has a path of length $m^2+1$ as a subtree. Hence, it is not $m$-burnable. Therefore, by Theorem~\ref{1103c}, $I^{sp}_{n,m}= m^2+n-2$.

Now, suppose $2\leq m\leq n$. Note that the unique $n$-spider with arm lengths $\underbrace{m,m, \dots, m}_{m \text{ times}}, \underbrace{1,1, \dots, 1}_{n-m \text{ times}}$ of order $m^2+n-m+1$ is the smallest (by order) $n$-spider containing the balanced $m$-spider of order $m^2+1$ as a subtree. Therefore, by Lemma~\ref{1804a} and Theorem~\ref{2702b}, 
$I^{sp}_{n,m}= m^2+n-m$.
\end{proof}

\section{Main Result on Path Forests and Its Application} 

In this section, we first prove our main result on path forests. Then we apply it to give an alternative proof of our Theorem~\ref{1103c}.

\begin{lemma}\label{1203a}
Let $m\geq 2$. If $T$ is a path forest with two paths and $\vert T\vert \leq m^2$, then $T$ is $m$-burnable unless the path orders of $T$ are $m^2-2$ and $2$.
\end{lemma}

\begin{proof}
We argue by induction on $m$. The base step $m=2$ is straightforward. For the induction step, suppose $m>2$ and the lemma is true for $m-1$. Now we may assume $T$ is a path forest of order $m^2$ with two paths of path orders $l_1\geq l_2$.

Since $m>2$, it follows that $m^2 >2(2m-2)$, thus implying that $l_1\geq 2m-1$. Note that $l_1=2m-1$ can hold only when $m=3$ and in this case, $T$ is clearly $3$-burnable. 

Suppose $l_1>2m-1$ and let $l'_1=l_1-(2m-1)$. Let $T'$ be the path forest with path orders $l'_1$ and $l_2$. Since $l'_1+l_2= m^2-(2m-1)=(m-1)^2$, by the induction hypothesis, $T'$ is $(m-1)$-burnable unless 
$\{l'_1, l_2\}=\{(m-1)^2-2, 2\}$. So if
$(l_1, l_2)\neq (m^2-2, 2)$ and $(l_1,l_2)\neq (2m+1, (m-1)^2-2)$, we can deduce that $T$ is $m$-burnable (by placing the first burning source at the first path of $T$ accordingly, and then follow the burning sequence of $T'$).
The case $(l_1, l_2)=(2m+1, (m-1)^2-2)$ can only happen when $m\in \{3,4\}$, that is, when $(l_1, l_2)=(7,2)$ or $(l_1,l_2)=(9,7)$. The special case $T$ with $(l_1,l_2)=(9,7)$ is clearly $4$-burnable.
\end{proof}

\begin{lemma}\label{2003d}
	Let $n\geq 2$ and suppose $T$ is a path forest with $n$ paths. If the order of $T$ is at most $3n-2$, then $T$ is $n$-burnable unless the smallest path order of $T$ is two.
\end{lemma}


\begin{proof}
	We argue by induction on $n$. The base step $n=2$ is straightforward.
	For the induction step, suppose $n>2$ and the lemma is true for $n-1$.
	Let $T$ be a path forest with path orders	
	$l_1\geq l_2\geq \dotsb \geq l_n \geq 1$ such that $l_1+l_2+\dotsb +l_n\leq 3n-2$. 
	It can be easily deduced that $3\leq l_1\leq 2n-1$.

	Suppose $l_n\neq 2$.
	Let $T'$ be the path forest obtained by deleting the first path of $T$.	
	Since $l_2+l_3+\dotsb +l_n\leq 3n-2-3= 3(n-1)-2$ and $l_n\neq 2$, by the induction hypothesis, $T'$ is $(n-1)$-burnable and thus $T$ is $n$-burnable (because $l_1\leq 2n-1$).	
	\end{proof}


\begin{theorem}\label{2003b}
Let $m\geq n\geq 2$ and suppose $T$ is a path forest with $n$ paths. If 
$$\vert T\vert \leq m^2-(n-1)^2+1,$$	
	then $T$ is $m$-burnable unless $\vert T\vert=m^2-(n-1)^2+1$ and
	$T$ is the unique path forest with path orders  $m^2-n^2+2, \underbrace{2,2,\dotsc, 2}_{n-1 \text{ times}}$.	
\end{theorem}

\begin{proof}
	 We argue by induction on $n$. The base step $n=2$ follows from Lemma~\ref{1203a}. For the induction step, suppose $n>2$ and the theorem is true for $n-1$ (and all $m\geq n-1$).
To prove that the theorem is true for all $m$ greater than or equal to this fixed $n$, we shall argue by induction on $m$. Suppose $m=n$ 
and so
$$\vert T \vert\leq m^2-(n-1)^2+1 =2n< 3n-2.$$
Then by Lemma~\ref{2003d}, $T$ is $n$-burnable
unless the smallest path order of $T$ is two, in which case  $T$ is the 
unique path forest with path orders $\underbrace{2,2,\dotsc, 2}_{n \text{ times}}$.

Now, suppose $m>n$ and the theorem is true for $m-1$ (for this fixed $n$). We may assume that $T$ is a path forest of order $m^2-(n-1)^2+1$ with path orders $l_1\geq l_2\geq \dotsb\geq l_n$ and 
$$(l_1, l_2, l_3, \dotsc, l_n)\neq (m^2-n^2+2,2,2,\dotsc, 2). \qquad (\ast)$$
We consider two cases.
	
\setcounter{case}{0}

\begin{case}
	$l_1\geq 2m$.
	
	Let $T'$ be the path forest with path orders $l_1-(2m-1), l_2, l_3, \dotsc, l_n$. Note that \mbox{$\vert T'\vert = (m-1)^2-(n-1)^2+1$}.
		By the induction hypothesis, $T'$ is  \mbox{$(m-1)$-burnable} and thus $T$ is $m$-burnable unless the path orders of $T'$ are 
	$(m-1)^2-n^2+2,  \underbrace{2,2,\dotsc, 2}_{n-1 \text{ times}}$.
	Hence, by assumption $(\ast)$, it follows that $l_1= 2m+1$, $l_2=(m-1)^2-n^2+2$, and
	$l_3=l_4=\dotsb =l_n=2$. Since $m>n\geq 3$, it can be verified that $2m+1\geq (m-1)^2-n^2+2$ is only possible when $m=n+1$ or $m=n+2$.	
	However, the case $m=n+1$ violates assumption $(\ast)$. Therefore, 
	it remains to see that the path forest $T$ with path orders
	$2n+5, 2n+3, \underbrace{2,2,\dotsc, 2}_{n-2 \text{ times}}$ is $(n+2)$-burnable. Indeed, since $(2n-1)+(2n+1)=4n>2n+5$, this path forest can be burned in $n+2$ rounds by placing the first burning source at the center of the second path of $T$, and the second and the third burning sources at the first path of $T$.	
\end{case}	
	
\begin{case}
	$l_1\leq 2m-1$.
	
	Note that
		$$m^2-(n-1)^2+1- 2n(m-n+1) = m^2-2mn+n^2=(m-n)^2> 0.$$
	It follows that $l_1> 2(m-n+1)$. Let $T''$ be the path forest with path orders $l_2, l_3, \dotsc, l_n$. Then
$$\vert T''\vert <{} m^2-(n-1)^2+1-2(m-n+1)
	=  (m-1)^2-(n-2)^2+1.$$
Therefore, since the theorem is true for $n-1$ by the first induction hypothesis, we know $T''$ is $(m-1)$-burnable and thus $T$ is $m$-burnable (because $l_1\leq 2m-1$).
	\end{case}	
The proof is complete.
\end{proof}

Our Second Main Theorem follows as a corollary.

Using our main result (Theorem~\ref{2003b}) on path forests, we give an alternative proof of Theorem~\ref{1103c}. One difference is that the series of proofs given in Section~\ref{0701a} only uses fairly weak technical results on path forests (Lemmas~\ref{0405a} and \ref{0603c}). The presentation of both proofs are provided so that it may give the readers a different perspectives and insight. We now restate Theorem~\ref{1103c} with an additional clause.

\begin{theorem}\label{2003a}
	Let $m>n\geq 2$. Then every $n$-spider of order at most $ m^2+n-2$ is $m$-burnable. Furthermore, if $l$ is the  length of the shortest arm, then the graph can be burned in $m$ rounds in such a way that after the head of the spider is burned, there are still at least $\min\{l,m-2\}$ rounds.	
\end{theorem}

\begin{proof}
We argue by induction on $n$. For the base step $n=2$, let $m>n$ and suppose  $T$ is a $2$-spider of order at most $m^2$ with arm lengths $l_1\geq l_2$.
If $l_2\leq m-1$, then by placing the first burning source at the $(m-1-l_2)$-th vertex of the first arm, it is clear that this $2$-spider can be burned in $m$ rounds where the entire second arm is burned by this first burning source. Suppose $l_2\geq m$ and $l_2\neq m+1$. If we place the first burning source at the head of the spider, the first $m-1$ vertices of each arm would be burned in $m$ rounds. The remaining $(m-1)^2$ vertices
unburned by the first source form two independent paths, neither which has order two. Hence, by Lemma~\ref{1203a}, this path forest is $(m-1)$-burnable and thus $T$ is $m$-burnable. The only remaining case for the $2$-spider is when $l_2=m+1$. In such a scenario, it is straightforward to check that there is a burning process of $m$ rounds where the first burning source is placed at the first vertex of the second arm. Therefore, the base step is complete.

For the induction step, suppose $n>2$ and the theorem is true for $n-1$. Let $m>n$ and suppose $T$ is an $n$-spider of order at most $m^2+n-2$ with arm lengths $l_1\geq l_2\geq \dotsb\geq l_n$.

\begin{claim}
We may assume that $l_n\geq m-1$.	
\end{claim}

\begin{proof}[Proof of the claim]
Suppose $l_n\leq m-2$.
 Let $T'$ be the $(n-1)$-spider obtained by deleting the $n$-th arm of $T$. 
 Since $\vert T'\vert \leq m^2+(n-1)-2$,
 by the induction hypothesis, $T'$ can be burned in $m$ rounds in such a way that after the head of the spider $T'$ is burned, there are still at least $\min\{l_{n-1}, m-2  \}$ rounds. The same burning sequence can clearly burn $T$ because  $\min\{l_{n-1}, m-2  \}\geq l_n$.\renewcommand{\qedsymbol}{}
\end{proof}

With the claim, we now have $l_i\geq m-1$ for all $1\leq i \leq n$. We will show that except for some special cases, $T$ is $m$-burnable with the first burning source being the head and thus there are $m-1$ rounds left after the head is burned. Note that the first $m-1$ vertices in each arm will be burned by the first source in $m$ rounds. The remaining vertices form a path forest $T''$ with $k$-paths, where $1\leq k\leq n$, such that
$$\vert T''\vert \leq m^2+n-3-n(m-1)= m^2-mn+2n-3.$$
If $k=1$, then $T''$ is $(m-1)$-burnable because it can be shown that $\vert T''\vert < (m-1)^2$ (due to $m>n>2$). Suppose \mbox{$k\neq 1$}. 


\setcounter{case}{0}

\begin{case}
$m\geq n+2$. Then
\begin{align*}
\vert T''\vert \leq {}& m(m-n)+2(n-2)+1\\
\leq {}& m(m-n)+(m-n)(n-2)+1\\
={}& (m+n-2)(m-n)+1\\
={}& (m-1)^2- (n-1)^2+1\\
\leq {}& (m-1)^2- (k-1)^2+1.
\end{align*}
By Theorem~\ref{2003b}, $T''$ is $(m-1)$-burnable and therefore $T$ is $m$-burnable  unless 
$\vert T''\vert=(m-1)^2- (k-1)^2+1$
and the path orders of $T''$ are $(m-1)^2-k^2+2, \underbrace{2,2,\dotsc, 2}_{k-1 \text{ times}}$. Note that this exceptional $T''$ is possible only 
 when $k=n$ and $m=n+2$. This leads us to the following special case.

\setcounter{special case}{0}

\begin{special case}
$m=n+2$, $l_1= 3n+4$, and $l_2=l_3=\dotsb = l_n=n+3$.

We place the first burning source at the first vertex of the $n$-th arm and thus in a burning process of $m$ rounds, there are $m-2$ rounds left after the head is burned.
Now, this first source would  burn the first $n$ vertices of each of the first $n-1$ arms and the first $n+2$ vertices of the $n$-th arm in $n+2$ rounds.  Hence, the remaining vertices form $n$ independent paths of orders $2n+4,  \underbrace{3,3,\dotsc, 3}_{n-2 \text{ times}}, 1$.
It is easy to see that this path forest can be burned with the remaining $n+1$ burning sources since $(2n-1)+(2n+1)=4n>2n+4$.
\end{special case}
\end{case}

\begin{case}
$m=n+1$. 

Then $\vert T''\vert\leq 3n-2$.
First, suppose $2\leq k \leq n-1$.
Then $$n^2-(k-1)^2+1\geq n^2-(n-2)^2+1=4n-3>3n-2$$ and thus the order of
$T''$ is strictly less than $n^2-(k-1)^2+1$. Hence, by Theorem~\ref{2003b}, $T''$ 
is $n$-burnable and thus $T$ is \mbox{$m$-burnable}. Now, suppose $k=n$.
Then by Lemma~\ref{2003d}, $T''$ is $n$-burnable and thus $T$ is $m$-burnable unless the smallest path order of $T''$ is two. This leads us to the following special case.

\begin{special case}
$m=n+1$ and $l_n=n+2$.

As in Special Case~1, we place the first burning source at the first vertex of the $n$-th arm. This first source would burn the first $n-1$ vertices of each of the first $n-1$ arms and the first $n+1$ vertices of the $n$-th arm in $n+1$ rounds.
Hence, the remaining vertices form $n$ independent paths of orders $l_1'\geq  l_2'\geq \dotsb\geq l_n'$ where
$$l_1'+ l_2'+\dotsb + l_n'\leq 4n-4= (n+1)^2+n-2-(n-1)(n-1)-(n+1)-1,$$
$l'_n=1$ and $l_{n-1}'\geq 3$. Hence, by Lemma~\ref{0405a}, this path forest is $n$-burnable and thus $T$ is $(n+1)$-burnable.

\end{special case}
\end{case}
The proof is now complete.
\end{proof}

\section{Conclusion}

This work is initiated based on an intuition that the more deviation from a path a tree is, the less number of rounds is needed to completely burn it. This is reflected in our main result on spiders: the number of arms 
determines how much larger the order of a spider can go beyond $m^2$ before being $m$-burnable is not guaranteed. Based on this finding, we conceive that the number of leaves 
of a tree is one deciding characteristic of its burning number. As a future direction, we would pursue the following conjecture.

\begin{conjecture} Suppose $m> n$.	
	 Every tree with $n$ leaves of order at most $m^2+n-2$ is $m$-burnable.	
\end{conjecture}

As a conclusion, this work presents a new perspectives towards the burning number conjecture. Although what we aim is stronger than the burning number conjecture, as shown by this work, this new approach may be more effective and feasible as utilization of the right characteristics of graphs may turn out to be the key in understanding burning numbers.

\section{Acknowledgment}

This work is completed during the sabbatical leave of the second author from 15 Nov 2018 until 14 Aug 2019, supported by Universiti Sains  Malaysia. He would like to extend his gratitude to Institute of Mathematical Sciences, University of Malaya for hosting his sabbatical leave.


\end{document}